\documentclass[review]{elsarticle}
\usepackage{lineno,hyperref}
\usepackage{amsfonts}
\usepackage{amssymb,amsmath,latexsym}
\usepackage{amsthm}
\usepackage{graphicx}
\usepackage{epstopdf}
\usepackage[table]{xcolor}
\usepackage{eurosym}
\usepackage{ragged2e}
\usepackage{float}
\usepackage{graphics}
\usepackage{caption}
\usepackage{subcaption}
\usepackage{natbib}

\modulolinenumbers[5]
\journal{arXiv.org}
\biboptions{sort&compress}
\ifpdf
\DeclareGraphicsExtensions{.eps,.pdf,.png,.jpg}
\else
\DeclareGraphicsExtensions{.eps}
\fi
\newtheorem{theorem}{Theorem}

\newtheorem{corollary}{Corollary}

\newtheorem{remark}{Remark}

\numberwithin{equation}{section}

\input{tcilatex}

\begin{document}

\begin{frontmatter}

\title{Generating functions for finite sums involving higher powers of binomial coefficients: Analysis of hypergeometric functions including new families of polynomials and numbers}

\author{Yilmaz Simsek\corref{mycorrespondingauthor}}
\address{Department of Mathematics, Faculty of Science University of Akdeniz TR-07058 Antalya, Turkey}
\cortext[mycorrespondingauthor]{Corresponding author}
\ead{ysimsek@akdeniz.edu.tr}

\begin{abstract}
The origin of this study is based on not only explicit formulas
of finite sums involving higher powers of binomial coefficients, but also
explicit evaluations of generating functions for this sums. It should be
emphasized that this study contains both new results and literature surveys about some of the related results that have existed so far. With the aid of hypergeometric function,
generating functions for a new family of the combinatorial numbers, related
to finite sums of powers of binomial coefficients, are constructed. By
using these generating functions, a number of new identities have been
obtained and at the same time previously well-known formulas and identities
have been generalized. Moreover, on this occasion, we identify new families of polynomials including some
families of well-known numbers such as Bernoulli numbers, Euler numbers,
Stirling numbers, Franel numbers, Catalan numbers, Changhee numbers, Daehee
numbers and the others, and also for the polynomials such as the Legendre
polynomials, Michael Vowe polynomial, the Mirimanoff polynomial, Golombek
type polynomials, and the others. We also give both Riemann and $p$-adic integral representations of these polynomials. Finally,
we give combinatorial interpretations of these new families of numbers,
polynomials and finite sums of the powers of binomial coefficients. We
also give open questions for ordinary generating functions for these numbers.
\end{abstract}


\end{frontmatter}

\linenumbers

\section{Introduction}

The studies on the finite sums of binomial coefficients and their powers
have been studied for long ages. These areas attracted the attention of many
mathematicians, statisticians, physicists and engineers. Therefore, numerous
studies have been published in the field of sums of powers of binomial
coefficients in recent years with different methods and techniques. For
instance, very recently, with the aid of the mathematical model related to
finite sums of binomial coefficients with their powers, and the Franel
numbers and Brownian paths, Essam and Guttmann \cite{FizikESSAMe} have
studied real world problems associated watermelons with more than two
chains. Like this, there exist many different examples (\textit{cf}. \cite%
{Askey}-\cite{Vandivier}; see also the references cited in each of these
earlier works).

Golombek (\cite{golombek}, \cite{Golombek2}) gave the following novel
combinatorial sums of binomial coefficients:%
\begin{equation}
B(n,k)=\sum_{j=1}^{k}\left( 
\begin{array}{c}
k \\ 
j%
\end{array}
\right) j^{n}=\frac{d^{n}}{dt^{n}}\left( e^{t}+1\right) ^{k}\left\vert
_{t=0}\right. ,  \label{Gl}
\end{equation}%
$n$ and $k$ are a positive integer (\textit{cf}. \cite{golombek}) and%
\begin{equation*}
\sum\limits_{k=1}^{n}\left( 
\begin{array}{c}
n \\ 
k%
\end{array}
\right) ^{2}k^{m}=\frac{d^{m}}{dt^{m}}\left\{ \sum\limits_{k=0}^{n}\left( 
\begin{array}{c}
n \\ 
k%
\end{array}
\right) ^{2}e^{tk}\right\} \left\vert _{t=0}\right. ,
\end{equation*}%
where $m$ is a nonnegative integer (\textit{cf}. \cite{Golombek2}). In \cite%
{golombek}, Golombek gave some comments and relations between the above sums
and the following sequences, respectively:%
\begin{equation*}
2^{n}\text{, }n2^{n-1}\text{, }n(n+1)2^{n-2}\text{, }\ldots
\end{equation*}%
and%
\begin{equation*}
\left( 
\begin{array}{c}
2n \\ 
n%
\end{array}
\right) -1\text{, }n\left( 
\begin{array}{c}
2n-1 \\ 
n%
\end{array}
\right) \text{, }n^{2}\left( 
\begin{array}{c}
2n-2 \\ 
n-1%
\end{array}
\right) \text{, }n^{2}(n+1)\left( 
\begin{array}{c}
2n-3 \\ 
n-1%
\end{array}
\right) \text{, }...\text{.}
\end{equation*}

On the light of the above finite sums including binomial coefficients with
their powers and hypergeometric functions, the main motivation of this paper
is to investigate some properties of the following sums with their
generating functions:

\begin{equation}
F_{y_{6}}(t,n;\lambda ,p)=\frac{1}{n!}\sum\limits_{k=0}^{n}\left( 
\begin{array}{c}
n \\ 
k%
\end{array}
\right) ^{p}\lambda ^{k}e^{tk}  \label{y6a}
\end{equation}%
and 
\begin{equation}
P(x;m,n;\lambda ,p)=\sum\limits_{j=0}^{n}\left( 
\begin{array}{c}
n \\ 
j%
\end{array}
\right) ^{p}\lambda ^{j}\left( x+j\right) ^{m},  \label{Yp2}
\end{equation}%
where $m$ and $p$ are nonnegative integers.

We shall give detail investigations, studies and remarks on the above novel
finite sums in the following sections

\subsection{Notations}

Before giving the main results of this article, it will be useful to give
some special notations, formulas and generating functions.

In this paper, $\mathbb{C}$, $\mathbb{R}$, $\mathbb{Q}$, $\mathbb{Z}$, and $%
\mathbb{N}$ are the sets of complex numbers, real numbers, rational numbers,
integers, and positive integers, respectively. $\mathbb{N}_{0}=\{0,1,2,3,$%
\ldots $\}$ and $\mathbb{Z}_{p}$ denotes the set of $p$-adic integers.

The principal value $\ln z$ is the logarithm whose imaginary part lies in
the interval $(-\pi ,\pi ]$. Moreover%
\begin{equation*}
0^{n}=\left\{ 
\begin{array}{cc}
1, & (n=0) \\ 
0, & (n\in \mathbb{N)}%
\end{array}
\right.
\end{equation*}%
and the Pochhammer's symbol for the rising factorial is defined by the
following notation%
\begin{equation*}
\left( \lambda \right) _{v}=\frac{\Gamma \left( \lambda +v\right) }{\Gamma
\left( \lambda \right) }=\lambda (\lambda +1)\cdots (\lambda +v-1),
\end{equation*}%
and%
\begin{equation*}
\left( \lambda \right) _{0}=1
\end{equation*}%
for $\lambda \neq 1$, where $v\in \mathbb{N}$, $\lambda \in \mathbb{C}$ and $%
\Gamma \left( \lambda \right) $ denotes the (Euler) gamma function%
\begin{equation*}
\Gamma \left( \lambda \right) =\int\limits_{0}^{\infty }t^{\lambda
-1}e^{-t}dt.
\end{equation*}

The above improper integral exists for complex variable $\lambda $ with $%
Re(\lambda )>0$ (or, if one prefers only to think of real variables, for
real $\lambda >0$).

\begin{equation*}
\left( 
\begin{array}{c}
z \\ 
v%
\end{array}
\right) =\frac{z(z-1)\cdots (z-v+1)}{v!}=\frac{\left( z\right) ^{\underline{
_{v}}}}{v!}\text{ }(v\in \mathbb{N}\text{, }z\in \mathbb{C)}
\end{equation*}%
and%
\begin{equation*}
\left( 
\begin{array}{c}
z \\ 
0%
\end{array}
\right) =1.
\end{equation*}%
We note that%
\begin{equation*}
\left( -\lambda \right) _{v}=(-1)^{v}\left( \lambda \right) ^{\underline{%
_{v} }}.
\end{equation*}%
The following references can be given for briefly the above notations (%
\textit{cf}. \cite{Askey}-\cite{WolframKOEp}).

\subsection{Generating functions for the special polynomials and numbers}

The Bernoulli polynomials of order $k$ are defined by%
\begin{equation}
F_{B}(t,x;k)=\left( \frac{t}{e^{t}-1}\right) ^{k}e^{tx}=\sum_{n=0}^{\infty
}B_{n}^{(k)}(x)\frac{t^{n}}{n!}\text{ }(\left\vert t\right\vert <2\pi )
\label{BerPol.}
\end{equation}%
where $k\in \mathbb{Z}$. Observe that $B_{n}^{(1)}(x)=B_{n}(x)$ denotes the
Bernoulli polynomials. When $k=0$, we easily see that $B_{n}^{(0)}(x)=x^{n}$%
. When $x=0$, we easily see that $B_{n}^{(k)}=B_{n}^{(k)}(0)$, which denotes
the Bernoulli numbers of order $k$\ (\textit{cf}. \cite{Choi}, \cite{Comtet}%
, \cite{Grademir}, \cite{Moll2012}, \cite{Ozden}, \cite{SimsekFPTA}, \cite%
{SimsekAADM}, \cite{SrivastavaBook}, \cite{srivas18}; see also the
references cited in each of these earlier works).

The Euler polynomials of order $k$, $E_{n}^{(k)}(x)$ are defined by means of
the following generating function:%
\begin{equation}
F_{E}(t,x;k)=\left( \frac{2}{e^{t}+1}\right) ^{k}e^{tx}=\sum_{n=0}^{\infty
}E_{n}^{(k)}(x)\frac{t^{n}}{n!},\text{ }(\left\vert t\right\vert <\pi )
\label{Eul.Pol}
\end{equation}%
where $k\in \mathbb{Z}$. When $k=0$, we have $E_{n}^{(0)}(x)=x^{n}$. Observe
that $E_{n}^{(1)}(x)=E_{n}(x)$ denotes the Euler polynomials. When $x=0$, we
have $E_{n}^{(k)}=E_{n}^{(k)}(0)$, which denoteS the Euler numbers of order $%
k$ (\textit{cf}. \cite{Comtet}, \cite{Grademir}, \cite{MollJIU}, \cite{Ozden}%
, \cite{SimsekFPTA}, \cite{SimsekAADM}, \cite{SrivastavaBook}, \cite%
{srivas18}; see also the references cited in each of these earlier works).

The Apostol-Bernoulli numbers and polynomials are defined by%
\begin{equation}
F_{B}\left( t,x;\lambda \right) =\frac{te^{xt}}{\lambda e^{t}-1}
=\sum\limits_{n=0}^{\infty }\mathcal{B}_{n}\left( x;\lambda \right) \frac{
t^{n}}{n!},  \label{B1}
\end{equation}%
(\textit{cf}. \cite{Apostol}, \cite{Grademir}, \cite{KimDkimYsim}, \cite%
{Kucukoglu}, \cite{Luo}, \cite{Ozden}, \cite{SimsekFPTA}, \cite{SimsekAADM}, 
\cite{TJM2018}, \cite{srivas18}, \cite{SrivastavaChoi2012}).

The Apostol-Euler numbers and polynomials are defined by%
\begin{equation}
F_{E}\left( t,x;\lambda \right) =\frac{2e^{xt}}{\lambda e^{t}+1}
=\sum\limits_{n=0}^{\infty }\mathcal{E}_{n}\left( x;\lambda \right) \frac{
t^{n}}{n!},  \label{E1}
\end{equation}%
(\textit{cf}. \cite{Grademir}, \cite{KimDkimYsim}, \cite{Kucukoglu}, \cite%
{Luo}, \cite{Ozden}, \cite{SimsekFPTA}, \cite{SimsekAADM}, \cite{TJM2018}, 
\cite{srivas18}, \cite{SrivastavaChoi2012}).

We now give generating functions for the Stirling numbers, which count the
number of ways to split a set of $n$ elements into $k$ nonempty parts; and
many others (\textit{cf}. \cite{Charamb}, \cite{Moll2012} \cite{Comtet}, 
\cite{Grademir}, \cite{SrivastavaBook}).

Let $v\in \mathbb{N}_{0}$. The Stirling numbers of the second kind are
defined by%
\begin{equation}
F_{S}(t,v)=\frac{\left( e^{t}-1\right) ^{v}}{v!}=\sum_{n=0}^{\infty }S(n,v) 
\frac{t^{n}}{n!}.  \label{SN-1}
\end{equation}%
and%
\begin{equation*}
x^{n}=\sum_{v=0}^{n}\left( 
\begin{array}{c}
x \\ 
v%
\end{array}
\right) v!S(n,v)
\end{equation*}%
(\textit{cf}. \cite{Charamb}, \cite{Moll2012} \cite{Cakic}, \cite{Comtet}, 
\cite{Grademir}, \cite{SrivastavaBook}; see also the references cited in
each of these earlier works).

Let $k\in \mathbb{N}_{0}$. The Stirling numbers of the first kind are
defined by%
\begin{equation}
F_{S_{1}}\left( t,k\right) =\frac{\left( \log \left( 1+t\right) \right) ^{k} 
}{k!}=\sum\limits_{n=0}^{\infty }s\left( n,k\right) \frac{t^{n}}{n!}
\label{S1}
\end{equation}%
(\textit{cf}. \cite{Charamb}, \cite{Moll2012} \cite{Cakic}, \cite{Comtet}, 
\cite{Grademir}, \cite{SrivastavaBook}; see also the references cited in
each of these earlier works).

Let $k\in \mathbb{N}_{0}$ and $\lambda \in \mathbb{C}$. We \cite{SimsekAADM}
defined the numbers $y_{1}(n,k;\lambda )$ by the following generating
function:%
\begin{equation}
F_{y_{1}}(t,k;\lambda )=\frac{1}{k!}\left( \lambda e^{t}+1\right)
^{k}=\sum_{n=0}^{\infty }y_{1}(n,k;\lambda )\frac{t^{n}}{n!}.  \label{ay1}
\end{equation}%
By using the above function, we have 
\begin{equation}
y_{1}(n,k;\lambda )=\frac{1}{k!}\sum_{j=0}^{k}\left( 
\begin{array}{c}
k \\ 
j%
\end{array}
\right) j^{n}\lambda ^{j}  \label{ay2}
\end{equation}%
where $n\in \mathbb{N}_{0}$ (\textit{cf}. \cite{SimsekAADM}).

The numbers $Y_{n}\left( \lambda \right) $ are defined by%
\begin{equation*}
F_{Y}\left( t,x;\lambda \right) =\frac{2}{\lambda ^{2}t+\lambda -1}
=\sum\limits_{n=0}^{\infty }Y_{n}\left( \lambda \right) \frac{t^{n}}{n!}
\end{equation*}
(\textit{cf}. \cite[Eq-(2.13)]{TJM2018}).

Generalized hypergeometric function $_{p}F_{p}$ can be written in the form%
\begin{equation*}
_{p}F_{q}\left[ 
\begin{array}{c}
\alpha _{1},...,\alpha _{p} \\ 
\beta _{1},...,\beta _{q}%
\end{array}
;z\right] =\sum\limits_{m=0}^{\infty }\left( \frac{\prod\limits_{j=1}^{p}
\left( \alpha _{j}\right) _{m}}{\prod\limits_{j=1}^{q}\left( \beta
_{j}\right) _{m}}\right) \frac{z^{m}}{m!}.
\end{equation*}%
The above series is also written as follows%
\begin{equation*}
_{p}F_{q}\left[ 
\begin{array}{c}
\alpha _{1},...,\alpha _{p} \\ 
\beta _{1},...,\beta _{q}%
\end{array}
;z\right] =\text{ }_{p}F_{p}\left( \alpha _{1},...,\alpha _{p};\beta
_{1},...,\beta _{q};z\right) .
\end{equation*}%
The above series converges for all $z$ if $p<q+1$, and for $\left\vert
z\right\vert <1$ if $p=q+1$. For this series one can assumed that all
parameters have general values, real or complex, except for the $\beta _{j}$%
, $j=1,2,...,q$ none of which is equal to zero or to a negative integer. If $%
p=q+1$, then the series is absolutely convergent on the set $\left\{ z\in 
\mathbb{C}:\left\vert z\right\vert <1\right\} $ if%
\begin{equation*}
\func{Re}\left( \sum\limits_{j=1}^{q}\beta j-\sum\limits_{j=1}^{p}\alpha
j\right) >0.
\end{equation*}%
For instance,%
\begin{equation*}
_{0}F_{0}\left( z\right) =e^{z},
\end{equation*}

\begin{equation*}
_{2}F_{1}\left( \alpha _{1};\alpha _{2};\beta _{1};z\right)
=\sum\limits_{k=0}^{\infty }\frac{\left( \alpha _{1}\right) _{k}\left(
\alpha _{2}\right) _{k}}{\left( \beta _{1}\right) _{k}}\frac{z^{k}}{k!},
\end{equation*}%
and%
\begin{equation*}
_{1}F_{0}\left[ 
\begin{array}{c}
b \\ 
-%
\end{array}
;x \right] =\frac{1}{\left( 1-x\right) ^{b}}
\end{equation*}%
(\textit{cf}. \cite{Luke}, \cite{SrivastavaBook}, \cite{SrivastavaChoi2012}, 
\cite{Temme}, \cite{Trembly}, \cite{WolframKOEp}).

The rest of this paper contains the following sections:

In Section 2, we give brief survey of the numbers $B(n,k)$.

In Section 3, we give a generating function for sums of finite sums of the
binomial coefficients with higher powers. We define a new family of
combinatorial sums, $y_{6}(m,n;\lambda ,p)$ and generalization of the Franel
numbers. We investigate some properties of these functions and numbers. A
relation between the Legendre polynomials, Stirling numbers, Changhee
numbers, Daehee numbers and the numbers $y_{6}(m,n;\lambda ,p)$ are given.
Remarks and comments on these functions and numbers are given.

In Section 4, we define a new family of special polynomials whose
coefficients are the numbers $y_{6}(m,n;\lambda ,p)$. We give not only the
derivative formulas for these polynomials, but also Riemann integral and $p$%
-adic integral formulas of these polynomials. Relations between these
polynomials, Mirimanoff polynomial, Frobenius-Euler, Bernoulli polynomials,
the Michael Vowe polynomials and the Legendre polynomials are given.

In Section 5, we give ordinary generating functions for the numbers $%
y_{6}\left( m,n;\lambda ;p\right) $. We also give some questions and
comments on these ordinary generating functions.

In Section 6, with the aid of functional equations for the generating
functions of the special numbers and polynomials, we derive some identities
and relations including the Bernoulli polynomials of order $k$, the Euler
polynomials of order $k$, the Stirling numbers, and the numbers $%
y_{6}(m,n;\lambda ,p)$.

\section{Brief survey of the numbers $B(n,k)$}

There are many combinatorial applications for (\ref{ay2}). Substituting $%
\lambda =1$ into (\ref{ay2}), yields 
\begin{equation}
B(n,k)=k!y_{1}(n,k;1).  \label{CC2}
\end{equation}

When we look at books, articles and other written documents about the sums
of the powers of binomial coefficients in the literature so far, in addition
to primarily book containing interesting results, written by Boas and Moll 
\cite{Boros Moll}, we were able to come across to other resources.
Meanwhile, we see that very interesting formulas and identities associated
with the binomial coefficients were given by Boas and Moll in their book 
\cite{Boros Moll}. In particular, there are some formulas and questions
containing not only the numbers $B(n,k)$, which are given below, but also
the other numbers. For instance, Boas and Moll \cite[Exercise 1.4.6.]{Boros
Moll} gave some exercises on the $B(0,k)$, $B(1,k)$, $B(2,k)$ and $B(6,k)$.
They also raised the following question as in \cite[Project 1.4.1.]{Boros
Moll}:

$a)$\textit{\ Use a symbolic language to observe that} 
\begin{equation*}
B(n,k)=2^{k-n}T_{n}(k)
\end{equation*}%
\textit{where }$T_{n}(k)$\textit{\ is a polynomial in }$k$\textit{\ of
degree }$n$\textit{.}

$b)$\textit{\ Explore properties of the coefficients of }$T_{n}(k)$\textit{.}

$c)$\textit{\ What can you say about the factors of }$T_{n}(k)$\textit{? The
factorization of a polynomial can be accomplished by the Mathematica command
Factor \textup{\cite[p. 15]{Boros Moll}}.}

Besides the above questions, Boas and Moll \cite[Exercise 1.4.6. (b)]{Boros
Moll} also gave the following question:

\textit{Determine formulas for the values of the alternating sums}%
\begin{equation*}
\sum_{j=0}^{k}(-1)^{j}\left( 
\begin{array}{c}
k \\ 
j%
\end{array}
\right) j^{n},
\end{equation*}%
\textit{for }$n=0$\textit{, }$1$\textit{\ and }$2$\textit{. Make a general
conjecture.}

Furthermore, Spivey \cite[Identity 12]{Spevy} and Boyadzhiev \cite[p.4,
Eq-(7)]{Boyadzhiev} proved that the numbers $B(n,k)$ are the linear
combination of the Stirling numbers of the second kind, $S(m,j)$,
respectively:%
\begin{equation}
B(m,n)=\sum_{j=0}^{m}\left( 
\begin{array}{c}
n \\ 
j%
\end{array}
\right) j!2^{n-j}S(m,j)  \label{Bs-1}
\end{equation}%
and%
\begin{equation*}
\sum_{j=0}^{k}\left( 
\begin{array}{c}
k \\ 
j%
\end{array}
\right) j^{n}x^{j}=\sum_{j=0}^{n}\left( 
\begin{array}{c}
k \\ 
j%
\end{array}
\right) j!S(n,j)x^{j}(1+x)^{k-j}.
\end{equation*}

In \cite[Remark 1]{SimsekAADM}, we see that%
\begin{equation*}
\sum_{j=0}^{k}(-1)^{j}\left( 
\begin{array}{c}
k \\ 
j%
\end{array}
\right) j^{n}=(-1)^{k}k!S(n,k).
\end{equation*}

On the other hand, in \cite[Remark 2]{SimsekAADM}, we remarked that the
numbers $B(n,k)$ are of the following form%
\begin{equation*}
a_{k}2^{k}
\end{equation*}%
where $\left(a_{k}\right)$ is an integer sequence depend on $k$,
independently of the aforementioned recent studies.

In addition, the following recursive formulas for the numbers $B(m,n)$,
which are not noticed by us in the studies published before \cite{SimsekAADM}%
, were also given in \cite[p. 12]{SimsekAADM}.

Let $d\in \mathbb{N}$ and $m_{0}$, $m_{1}$, $m_{2}$,$...$, $m_{d}\in \mathbb{%
\ Q}$. Let $m_{0}\neq 0$. Then%
\begin{equation}
\sum_{v=0}^{d-1}m_{v}B(d-v,k)=2^{k-d}\left( 
\begin{array}{c}
k \\ 
d%
\end{array}
\right)  \label{CB1}
\end{equation}%
and%
\begin{equation*}
B(d,k)=\frac{2^{k-d}}{m_{0}}\left( 
\begin{array}{c}
k \\ 
d%
\end{array}
\right) -\sum_{v=1}^{d-1}\frac{m_{v}}{m_{0}}B(d-v,k).
\end{equation*}

By novel computation method including the well-known Faa di Bruno's formula,
Xu \cite[Eq-(27)]{Xu} proved that the numbers $m_{v}$\ are given by the
following formula:%
\begin{equation*}
m_{v}=\frac{1}{d!}s(d,d-v)
\end{equation*}%
where $d\geq 1$ and $v=0$, $1$, \ldots, $d$.

We have already mentioned in the above that Boas and Moll \cite{Boros Moll}
formulated the numbers $B(d,k)$ by the polynomials $T_{d}(k)$. We \cite%
{SimsekAADM} also gave the following another expanded form of the numbers $%
B(d,k)$:%
\begin{equation*}
B(d,k)=(k^{d}+x_{1}k^{d-1}+x_{2}k^{d-2}+\cdots
++x_{d-2}k^{2}+x_{d-1}k)2^{k-d},
\end{equation*}%
where $x_{1}$, $x_{2}$,$\ldots $, $x_{d-1}$ are integers and $d$ is a
positive integer.

Xu \cite[Eq-(27)]{Xu} proved that the coefficients $x_{1}$, $x_{2}$,$\ldots $%
, $x_{d-1}$ are computed by the following formula:%
\begin{equation*}
x_{d-l}=\sum_{j=l}^{d}s(j,l)S(d,j)2^{d-j}
\end{equation*}%
where $l=1$, $2$, \ldots, $d-1$.

On the other hand, in \cite{SimsekAADM}, we also gave the following further
question:

\textit{Is it possible to find $f_{d}(x)$ function which satisfy} 
\begin{equation*}
f_{d}(x)=\sum_{k=0}^{\infty }B(n,k)x^{k}?
\end{equation*}

In \cite{SimsekAADM}, we gave only the followings for the numbers $B(0,k)$
and $B(1,k)$:%
\begin{equation*}
\sum_{k=0}^{\infty }B(0,k)x^{k}=\frac{1}{1-2x}
\end{equation*}%
and%
\begin{equation*}
\sum_{k=1}^{\infty }B(1,k)x^{k}=\frac{x}{\left( 1-2x\right) ^{2}}
\end{equation*}%
where $\left\vert x\right\vert <\frac{1}{2}$. But, Xu \cite[Eq-(32)]{Xu}
proved that the functions $f_{d}(x)$ have the following form:%
\begin{equation*}
f_{d}(x)=\sum_{j=1}^{d}j!S(d,j)\frac{x^{j}}{\left( 1-2x\right) ^{j+1}},
\end{equation*}%
where $d\geq 1$.

In \cite[Eq-(28)-(29)]{SimsekAADM}, we gave the following relations:

\begin{equation}
E_{n}^{(-k)}(\lambda )=k!2^{-k}y_{1}(n,k;\lambda )  \label{Cab3}
\end{equation}%
where $k\in \mathbb{N}_{0}$ and $E_{n}^{(-k)}(\lambda )$ denotes the
Apostol-Euler numbers of order $-k$. Substituting $\lambda =1$ into the
above equation, we also have%
\begin{equation}
E_{n}^{(-k)}=2^{-k}B(n,k).  \label{Caa3}
\end{equation}

Finally, we complete this section by mentioning that maybe there are other
different papers or survey manuscripts on the numbers $B(d,k)$ and their
applications.

\section{Generating functions for finite sums of the Binomial coefficient
with higher powers}

In this section, we give explicit formulas of finite sums involving powers
of binomial coefficients with their generating functions. The type of these
sums are given by (\ref{y6a}) and (\ref{Yp2}). That is, for $m,p\in \mathbb{N%
}_{0}$, we shall study and investigate the following finite sums of the
binomial coefficient with higher powers: 
\begin{equation*}
\sum\limits_{j=0}^{n}\left( 
\begin{array}{c}
n \\ 
j%
\end{array}
\right) ^{p}\lambda ^{j}j^{m}.
\end{equation*}%
including the following the special case%
\begin{equation*}
\sum\limits_{j=0}^{n}\left( 
\begin{array}{c}
n \\ 
j%
\end{array}
\right) ^{p}j^{m}.
\end{equation*}%
The above finite sum has a long history. These sums have been studied in a
general framework by the works of Golombek \cite{golombek}-\cite{Golombek2},
Cusick \cite{Cusick} and Moll \cite{Moll2012}, and also \cite{Boros Moll}.

We construct the following generating function for a new family of new
numbers $y_{6}(n,k;\lambda ,p)$:%
\begin{equation}
F_{y_{6}}(t,m;\lambda ,p)=\frac{1}{n!}\text{ }_{p}F_{p-1}\left[ 
\begin{array}{c}
-n,-n,...,-n \\ 
1,1,...,1%
\end{array}
;\left( -1\right) ^{p}\lambda e^{t}\right] =\sum_{m=0}^{\infty
}y_{6}(m,n;\lambda ,p)\frac{t^{m}}{m!},  \label{y6}
\end{equation}%
where $n,p\in \mathbb{N}$ and $\lambda \in \mathbb{R}$ or $\mathbb{C}$.

Observe that equation (\ref{y6}) yields (\ref{y6a}).

\begin{remark}
Substituting $\lambda =1$ and $p=2$ into \textup{(\ref{y6a})}, we have 
\begin{equation*}
g\left( n,t\right) =\sum\limits_{k=0}^{n}\left( 
\begin{array}{c}
n \\ 
k%
\end{array}
\right) ^{2}e^{tk}=n!F_{y_{6}}(t,n;1,2)
\end{equation*}
(\textit{cf}. \textup{\cite{Golombek2}}).
\end{remark}

Using Taylor series for $e^{tk}$, we obtain the following relation: 
\begin{equation*}
\sum_{m=0}^{\infty }y_{6}(m,n;\lambda ,p)\frac{t^{m}}{m!}=\sum_{m=0}^{\infty
}\left( \frac{1}{n!}\sum\limits_{k=0}^{n}\left( 
\begin{array}{c}
n \\ 
k%
\end{array}
\right) ^{p}k^{m}\lambda ^{k}\right) \frac{t^{m}}{m!}.
\end{equation*}%
Comparing the coefficients of $\frac{t^{m}}{m!}$ on both sides of the above
equation, we find that the numbers $y_{6}(m,n;\lambda ,p)$ have the
following explicit formula:

\begin{theorem}
Let $n,m,p\in \mathbb{N}_{0}$. The following identity holds:  
\begin{equation}
y_{6}(m,n;\lambda ,p)=\frac{1}{n!}\sum\limits_{k=0}^{n}\left( 
\begin{array}{c}
n \\ 
k%
\end{array}
\right) ^{p}k^{m}\lambda ^{k}.  \label{y6b}
\end{equation}
\end{theorem}

Taking $m$ times derivative of (\ref{y6}), with respect to $t$, we get the
following formula for the numbers $y_{6}(m,n;\lambda ,p)$:

\begin{corollary}
Let $m\in \mathbb{N}$. The we have 
\begin{equation}
y_{6}(m,n;\lambda ,p)=\frac{\partial ^{m}}{\partial t^{m}}\left\{
F_{y_{6}}(t,m,n;\lambda ,p)\right\} \left\vert _{t=0}\right. .  \label{y6G}
\end{equation}
\end{corollary}

\begin{remark}
Substituting $\lambda =1$ into \textup{(\ref{y6b})}, we have the following
well-known  result: 
\begin{equation*}
n!y_{6}(m,n;1,p)=M_{m,p}(n)=\sum\limits_{k=0}^{n}\left( 
\begin{array}{c}
n \\ 
k%
\end{array}
\right) ^{p}k^{m}
\end{equation*}
(\textit{cf}. \textup{\cite[p. 159, Eq-(5.1.1)]{Moll2012})}; see also 
\textup{\cite{Cusick}}  and \textup{\cite{Golombek2}}. Substituting $\lambda
=1$ and $p=2$ \textup{(\ref{y6b})} and  \textup{(\ref{y6G})}, we have the
following well-known results: 
\begin{equation*}
n!y_{6}(m,n;1,2)=Q(n,m)=\frac{\partial ^{m}}{\partial t^{m}}\left\{
F_{y_{6}}(t,m,n;1,2)\right\} \left\vert _{t=0}\right. ,
\end{equation*}%
and  
\begin{equation*}
\sum\limits_{k=1}^{n}\left( 
\begin{array}{c}
n \\ 
k%
\end{array}
\right) ^{2}k^{m}=\frac{\partial ^{m}}{\partial t^{m}}\left\{
\sum\limits_{k=0}^{n}\left( 
\begin{array}{c}
n \\ 
k%
\end{array}
\right) ^{2}e^{tk}\right\} \left\vert _{t=0}\right.
\end{equation*}
(\textit{cf}. \textup{\cite{Golombek2}}).
\end{remark}

\begin{remark}
When $\lambda =1$, equation \textup{(\ref{y6b})} also reduces to the Moments
sums $ M_{m,p}(n)$ which were defined by Moll \textup{\cite[p. 167,
Eq-(5.3.1)]{Moll2012}} . We  also observe that many recurrences for powers
of binomials including $ M_{m,p}(n)$ given by Moll \textup{\cite[p. 172]%
{Moll2012}} .
\end{remark}

By using equation (\ref{y6b}), we represent the numbers $y_{6}(0,n;\lambda
,p)$ by the following hypergeometric function as follows:

\begin{corollary}
\begin{equation}
y_{6}(0,n;\lambda ,p)=\frac{1}{n!}\text{ }_{p}F_{p-1}\left[ 
\begin{array}{c}
-n,-n,...,-n \\ 
1,1,...,1%
\end{array}
;\left( -1\right) ^{p}\lambda \right] .  \label{y6bb}
\end{equation}
\end{corollary}

\begin{remark}
Setting $\lambda =1$ and $p=2$ into \textup{(\ref{y6bb})}, we get the
following  well-known identity: 
\begin{equation*}
n!y_{6}(0,n;1,2)=\sum\limits_{k=0}^{n}\left( 
\begin{array}{c}
n \\ 
k%
\end{array}
\right) ^{2}=\text{ }_{2}F_{1}\left[ 
\begin{array}{c}
-n,-n \\ 
1%
\end{array}
;1\right] .
\end{equation*}
The above finite sum is obviously a particular case of the Chu-Vandermonde 
identity (cf. \textup{\cite[p. 37]{WolframKOEp}}). Setting $\lambda =1$ and $%
p=3$ and also replacing $n$ by $2n$  in \textup{(\ref{y6bb})}, we have 
\begin{eqnarray*}
\left( 2n\right) !y_{6}(0,2n;-1,3) &=&\sum\limits_{k=0}^{2n}(-1)^{k}\left( 
\begin{array}{c}
2n \\ 
k%
\end{array}
\right) ^{3}=\text{ }_{3}F_{2}\left[ 
\begin{array}{c}
-2n,-2n,-2n \\ 
1,1%
\end{array}
;1\right] \\
&=&(-1)^{n}\frac{(3n)!}{\left( n!\right) ^{3}}
\end{eqnarray*}
which is a special case of Dixon's identity (\textit{cf}. \textup{\cite[pp.
37-38, p. 64]{WolframKOEp}}).
\end{remark}

\begin{remark}
Substituting $\lambda =1$ into \textup{(\ref{y6b})}, Cusic \textup{\cite%
{Cusick}} gave a  recurrence relations for sum $S_{n,m}^{(p)}$: 
\begin{equation*}
S_{n,m}^{(p)}=n!y_{6}(m,n;1,p)=\sum\limits_{k=0}^{n}\left( 
\begin{array}{c}
n \\ 
k%
\end{array}
\right) ^{p}k^{m}=\sum\limits_{k=0}^{m}(-1)^{k}\left( 
\begin{array}{c}
m \\ 
k%
\end{array}
\right) n^{m-k}S_{n,k}^{(p)}
\end{equation*}
and 
\begin{equation*}
S_{n,p}^{(p)}=n^{p}S_{n,0}^{(p)},
\end{equation*}
where 
\begin{equation*}
S_{n,0}^{(p)}=n!y_{6}(0,n;1,p).
\end{equation*}
\end{remark}

\subsection{Generalization of the Franel numbers}

Here, we give generalization of the Franel numbers $F_{p}(m,n;\lambda )$. We
also give a relation between the Bernoulli numbers, the Stirling numbers of
the first kind, the Catalan numbers, the Daehee numbers and the numbers $%
y_{6}(m,n;\lambda ,p)$.

Generalization Franel numbers $F_{p}(m,n;\lambda )$ are defined by%
\begin{equation*}
F_{p}(m,n;\lambda )=n!y_{6}(m,n;\lambda ,p).
\end{equation*}%
Here, the number $F_{p}(m,n;\lambda )$ are also so-called \textit{%
generalized }$p$\textit{-th order Franel numbers}.

\begin{remark}
In 1894, Franel \textup{\cite{Franel}} gave a recurrence relation for $%
S_{n,0}^{(3)}$ (\textit{cf}. \textup{\cite{Cusick}, \cite{foat}, \cite%
{Perlstadt JNT}}).
\end{remark}

\begin{remark}
In \textup{\cite{Golombek2}}, we observe that  
\begin{equation*}
n!y_{6}(3,n;1,2)=n^{2}(n+1)\left( 
\begin{array}{c}
2n-3 \\ 
n-1%
\end{array}
\right)
\end{equation*}%
In \textup{\cite{Cusick}} and \textup{\cite{Askey}}, the  following
well-known relations were given: 
\begin{equation*}
v_{n}=n!y_{6}(0,n;1,3)=\sum\limits_{k=0}^{n}\left( 
\begin{array}{c}
n \\ 
k%
\end{array}
\right) ^{3}=\text{ }_{3}F_{2}\left[ 
\begin{array}{c}
-n,-n,-n \\ 
1,1%
\end{array}
;-1\right] ,
\end{equation*}
where the numbers $v_{n}$\ denote the well-known the \textbf{Franel numbers}
and also  
\begin{equation*}
u_{n}=n!y_{6}(0,n;1,4)=\sum\limits_{k=0}^{n}\left( 
\begin{array}{c}
n \\ 
k%
\end{array}
\right) ^{4}=\text{ }_{4}F_{3}\left[ 
\begin{array}{c}
-n,-n,-n,-n \\ 
1,1,1%
\end{array}
;1\right],
\end{equation*}
(\textit{cf}. \textup{\cite{Cusick}, \cite{Franel}, \cite{Franel2}, \cite%
{Moll2012}}).
\end{remark}

Substituting $m=0$, $\lambda =1$ and $p=2$ into (\ref{y6b}) yields 
\begin{equation*}
y_{6}(0,n;1,2)=\frac{n+1}{n!}C_{n},
\end{equation*}%
where $C_{n}$ denotes the Catalan numbers, which count the number of ways to
place parentheses to group symbols in a sequence of numbers (\textit{cf}. 
\cite{Moll2012}). We modify the above equation as follows%
\begin{equation}
y_{6}(0,n;1,2)=(-1)^{n}\frac{C_{n}}{D_{n}}  \label{CN}
\end{equation}%
and $D_{n}$ denotes the Daehee numbers (\textit{cf}. \cite{DSkim1}).
Combining (\ref{CN}) with the following well-known relation%
\begin{equation*}
D_{n}=\sum\limits_{k=0}^{n}B_{k}s(n,k),
\end{equation*}%
we arrive at the following theorem:

\begin{theorem}
\begin{equation*}
C_{n}=(-1)^{n}y_{6}(0,n;1,2)\sum\limits_{k=0}^{n}B_{k}s(n,k).
\end{equation*}
\end{theorem}

\begin{remark}
Substituting $p=1$ into \textup{(\ref{y6b})}, we have 
\begin{equation*}
y_{6}(m,n;\lambda ,1)=y_{1}(m,n;\lambda )
\end{equation*}
(\textit{cf}. \textup{\cite{SimsekAADM}}). Substituting $p=1$ and $\lambda =1
$ into \textup{(\ref{y6b})}, we have 
\begin{equation*}
B(m,n)=n!y_{6}(m,n;1,1)
\end{equation*}
(\textit{cf}. \textup{\cite{Boros Moll}, \cite{golombek}, \cite{Golombek2}, 
\cite{Moll2012}, \cite{SimsekAADM}}).
\end{remark}

Substituting $\lambda =-1$ into (\ref{y6b}), we have generalized the $p$-th
order alterne Franel numbers as follows:%
\begin{equation}
y_{6}(m,n;-1,p)=\frac{1}{n!}\sum\limits_{k=0}^{n}(-1)^{k}\left( 
\begin{array}{c}
n \\ 
k%
\end{array}
\right) ^{p}k^{m}.  \label{Y6.A}
\end{equation}%
Substituting $m=0$ into the above equation, we have%
\begin{equation}
y_{6}(0,n;-1,p)=\frac{1}{n!}\sum\limits_{k=0}^{n}(-1)^{k}\left( 
\begin{array}{c}
n \\ 
k%
\end{array}
\right) ^{p}.  \label{Y6B}
\end{equation}

\begin{remark}
Substituting $\lambda =1$ into \textup{(\ref{y6b})}, we have the following
power sums 
\begin{equation*}
L_{j}(n)=n!y_{6}(0,n;1,j)
\end{equation*}
(\textit{cf}. \textup{\cite[p. 160, Eq-(5.2.1)]{Moll2012}}). Moll gave 
\textup{\cite{Moll2012}}  recurrence relations for $L_{j}(n)$. Castro et al. 
\textup{\cite{Castro}} also studied  power sums of $L_{j}(n)$ which
so-called the $j$-th order Franel numbers and also the $j$ -th order alterne
Franel numbers, given as follows:  
\begin{equation*}
n!y_{6}(0,n;-1,j).
\end{equation*}
\end{remark}

\begin{remark}
Setting $p=2$ and $p=3$ into \textup{(\ref{Y6B})}, we have the following
well-known  identities, respectively: 
\begin{equation*}
n!y_{6}(0,n;-1,2)=\left\{ 
\begin{array}{cc}
0, & n\text{ odd} \\ 
\frac{(-1)^{\frac{n}{2}}n!}{\left( \left( \frac{n}{2}\right) !\right) ^{2}},
& \text{otherwise,}%
\end{array}
\right.
\end{equation*}
\begin{equation}
n!y_{6}(0,n;-1,2)=\frac{\sqrt{\pi }2^{n}}{\Gamma \left( \frac{2+n}{2}\right)
\Gamma \left( \frac{1-n}{2}\right) },  \label{AWolf}
\end{equation}
(\textit{cf}. \textup{\cite[p. 29, Exercise 2.7]{WolframKOEp}}), and 
\begin{equation*}
n!y_{6}(0,n;-1,3)=\left\{ 
\begin{array}{cc}
0, & n\text{ odd} \\ 
\frac{(-1)^{\frac{n}{2}}\left( \frac{3n}{2}\right) !}{\left( \left( \frac{n}{
2}\right) !\right) ^{3}}, & \text{otherwise}%
\end{array}
\right.
\end{equation*}
(\textit{cf}. \textup{\cite[p. 11]{WolframKOEp}}).
\end{remark}

\subsection{A relation between the numbers $y_{6}(m,n;\protect\lambda ,p)$,
Euler, Stirling and Changhee numbers, and the Legendre polynomials}

Here we give relations between the numbers $y_{6}(m,n;\lambda ,p)$ and the
Euler numbers, Stirling numbers and Changhee numbers, and also the Legendre
polynomials.

The Legendre polynomials $P_{n}\left( x\right) $, orthogonal in $\left[ -1,1%
\right] $, are defined by%
\begin{equation*}
g\left( t,x\right) =\frac{1}{\sqrt{1-2xt+t^{2}}}=\sum\limits_{n=0}^{\infty
}P_{n}\left( x\right) t^{n}
\end{equation*}%
where%
\begin{equation*}
P_{n}\left( x\right) =\frac{1}{2^{n}}\sum\limits_{k=0}^{n}\left( 
\begin{array}{c}
n \\ 
k%
\end{array}
\right) ^{2}\left( x-1\right) ^{n-k}\left( x+1\right) ^{k}
\end{equation*}%
(\textit{c}f. \cite{Moll2012}, \cite[p. xiv Eq-(2) and p. 73]{WolframKOEp}).

When $x=0$, we get the following identities%
\begin{eqnarray*}
P_{n}\left( 0\right) &=&\frac{\left( -1\right) ^{n}n!}{2^{n}}y_{6}(0,n;-1,2),
\\
P_{n}\left( 2\right) &=&\frac{n!}{2^{n}}y_{6}(0,n;3,2),
\end{eqnarray*}%
\begin{equation*}
P_{n}\left( 2\right) =-Y_{n}(-1)y_{6}(0,n;3,2),
\end{equation*}%
where the numbers $Y_{n}(\lambda )$ are given in , and%
\begin{equation*}
P_{n}\left( 0\right) =Ch_{n}y_{6}(0,n;-1,2),
\end{equation*}%
where $Ch_{n}$ denotes the Changhee numbers (\textit{cf}. \cite{DsKim2}).
Since%
\begin{equation*}
Ch_{n}=\frac{\left( -1\right) ^{n}n!}{2^{n}}=\sum
\limits_{k=0}^{n}s(n,k)E_{k},
\end{equation*}%
we get the following theorem:

\begin{theorem}
\begin{equation*}
P_{n}\left( 0\right) =y_{6}(0,n;-1,2)\sum\limits_{k=0}^{n}s(n,k)E_{k}.
\end{equation*}
\end{theorem}

\section{A new family of Polynomials $P(x;m,n;\protect\lambda ,p)$}

In this section we define a new family of special polynomials whose
coefficients are the numbers $y_{6}(m,n;\lambda ,p)$. We give some
properties of these polynomials.

The polynomials $P(x;m,n;\lambda ,p)$ are defined by means of the following
generating function:%
\begin{eqnarray*}
G(t,x;\lambda ,p) &=&\frac{1}{n!}\text{ }_{p}F_{p-1}\left[ 
\begin{array}{c}
-n,-n,...,-n \\ 
1,1,...,1%
\end{array}
;\left( -1\right) ^{p}\lambda e^{t}\right] \text{ }_{p}F_{p}\left[ 
\begin{array}{c}
1,1,...,1 \\ 
1,1,...,1%
\end{array}
;xt\right] \\
&=&\sum_{m=0}^{\infty }P(x;m,n;\lambda ,p)\frac{t^{m}}{m!}
\end{eqnarray*}%
where $n,p\in \mathbb{N}$ and $\lambda \in \mathbb{R}$ or $\mathbb{C}$.

From the above equation, we also get%
\begin{equation}
G(t,x;\lambda ,m,p)=e^{tx}F_{y_{6}}(t,m;\lambda ,p).  \label{py6}
\end{equation}%
By (\ref{py6}), we get%
\begin{equation*}
\sum_{m=0}^{\infty }P(x;m,n;\lambda ,p)\frac{t^{m}}{m!}=\sum_{m=0}^{\infty
}x^{m}\frac{t^{m}}{m!}\sum_{m=0}^{\infty }y_{6}(m,n;\lambda ,p)\frac{t^{m}}{
m!}.
\end{equation*}%
Therefore%
\begin{equation*}
\sum_{m=0}^{\infty }P(x;m,n;\lambda ,p)\frac{t^{m}}{m!}=\sum_{m=0}^{\infty
}\sum\limits_{k=0}^{m}\left( 
\begin{array}{c}
m \\ 
k%
\end{array}
\right) x^{m-k}y_{6}(k,n;\lambda ,p)\frac{t^{m}}{m!}
\end{equation*}%
Comparing the coefficients of $\frac{t^{m}}{m!}$ on both sides of the above
equation, we arrive at the following theorem:

\begin{theorem}
Let $m,n,p\in \mathbb{N}_{0}$ and $\lambda\in \mathbb{R}$ or $\mathbb{C}$.
Then we have{\large \ 
\begin{equation}
P(x;m,n;\lambda ,p)=\sum\limits_{k=0}^{m}\left( 
\begin{array}{c}
m \\ 
k%
\end{array}
\right) x^{m-k}y_{6}(k,n;\lambda ,p).  \label{Yp1}
\end{equation}
}
\end{theorem}

By (\ref{py6}), we also obtain another formula for the polynomials $%
P(x;m,n;\lambda ,p)$, which is given by (\ref{Yp2}).

Combining (\ref{Yp1}) with (\ref{Yp2}), we get the following combinatorial
sum:

\begin{corollary}
\begin{equation*}
\sum\limits_{k=0}^{m}\left( 
\begin{array}{c}
m \\ 
k%
\end{array}
\right) x^{m-k}y_{6}(k,n;\lambda ,p)=\sum\limits_{j=0}^{n}\left( 
\begin{array}{c}
n \\ 
j%
\end{array}
\right) ^{p}\lambda ^{j}\left( x+j\right) ^{m}.
\end{equation*}
\end{corollary}

\subsection{Derivative formula and recurrence relation}

In this section we are going to differentiate (\ref{py6}) with respect to $t$
and $x$ to derive a derivative formula and recurrence relations for the
polynomials $P(x;m,n;\lambda ,p)$.

We give higher-order derivatives formula of the polynomials $P(x;m,n;\lambda
,p)$, with respect to $x$ of equation (\ref{py6}), by the following partial
differential equation:%
\begin{equation*}
\frac{\partial ^{k}}{\partial x^{k}}\left\{ G(t,x;\lambda ,m,p)\right\}
=t^{k}G(t,x;\lambda ,m,p).
\end{equation*}%
By using the above equation, we obtain%
\begin{equation*}
\sum_{m=0}^{\infty }\frac{\partial ^{k}}{\partial x^{k}}\left\{
P(x;m,n;\lambda ,p)\right\} \frac{t^{m}}{m!}=\sum_{m=0}^{\infty }(n)^{ 
\underline{k}}P(x;m-k,n;\lambda ,p)\frac{t^{m}}{m!}.
\end{equation*}%
Comparing the coefficients of $\frac{t^{m}}{m!}$ on both sides of the above
equation, we arrive at the following theorem:

\begin{theorem}
Let $0<k\leq m$. Then we have 
\begin{equation}
\frac{\partial ^{k}}{\partial x^{k}}\left\{ P(x;m,n;\lambda ,p)\right\}
=(n)^{\underline{k}}P(x;m-k,n;\lambda ,p).  \label{py6a}
\end{equation}
\end{theorem}

Substituting $k=1$ into (\ref{py6a}), we get%
\begin{equation*}
\frac{\partial }{\partial x}\left\{ P(x;m,n;\lambda ,p)\right\}
=nP(x;m-1,n;\lambda ,p).
\end{equation*}

Combining (\ref{y6a}) with (\ref{py6}), we get%
\begin{equation*}
G(t,x;\lambda ,m,p)=\frac{1}{n!}\sum\limits_{k=0}^{n}\left( 
\begin{array}{c}
n \\ 
k%
\end{array}
\right) ^{p}\lambda ^{k}e^{t\left( k+x\right) }.
\end{equation*}%
In order to derive recurrence relation for the polynomials $P(x;m,n;\lambda
,p)$, we give the following partial differential equation of equation (\ref%
{py6}), with respect to $t$:%
\begin{equation*}
\frac{\partial }{\partial t}\left\{ G(t,x;\lambda ,m,p)\right\} =\frac{1}{n!}
\sum\limits_{k=0}^{n}\left( 
\begin{array}{c}
n \\ 
k%
\end{array}
\right) ^{p}\lambda ^{k}\left( k+x\right) e^{t\left( k+x\right) }
\end{equation*}%
Therefore%
\begin{eqnarray*}
\sum_{m=0}^{\infty }P(x;m+1,n;\lambda ,p)\frac{t^{m}}{m!} &=&\sum_{m=0}^{
\infty }\sum\limits_{j=0}^{m}\left( 
\begin{array}{c}
m \\ 
j%
\end{array}
\right) x^{m-j}y_{6}\left( j+1,n;\lambda ;p\right) \frac{t^{m}}{m!} \\
&&+x\sum_{m=0}^{\infty }P(x;m,n;\lambda ,p)\frac{t^{m}}{m!}
\end{eqnarray*}

Comparing the coefficients of $\frac{t^{m}}{m!}$ on both sides of the above
equation, we arrive at the following theorem:

\begin{theorem}
\begin{equation}
P(x;m+1,n;\lambda ,p)-xP(x;m,n;\lambda ,p)=\sum\limits_{j=0}^{m}\left( 
\begin{array}{c}
m \\ 
j%
\end{array}
\right) x^{m-j}y_{6}\left( j+1,n;\lambda ;p\right).  \label{py6ab}
\end{equation}
\end{theorem}

Here we note that it may be possible to express the relation on the right
side of equation (\ref{py6ab}) with derivative of the polynomials $%
P(x;m,n;\lambda ,p)$

\subsection{Integrals of the polynomials $P(x;m,n;\protect\lambda ,p)$}

Here we give not only the Riemann integral, but also $p$-adic integrals of
the polynomials $P(x;m,n;\lambda ,p)$.

\subsubsection{Riemann integral of the polynomials $P(x;m,n;\protect\lambda %
,p)$}

Here, we give some identities with aid of the Riemann integral of the
polynomials $P(x;m,n;\lambda ,p)$.

Integrating both sides of (\ref{Yp1}) and (\ref{Yp2}) with respect to $x$
yields, respectively,%
\begin{equation}
\int\limits_{0}^{1}P(x;m,n;\lambda ,p)dx=\sum\limits_{k=0}^{m}\left( 
\begin{array}{c}
m \\ 
k%
\end{array}
\right) \frac{y_{6}(k,n;\lambda ,p)}{m-k+1}  \label{inP-1}
\end{equation}%
and%
\begin{equation}
\int\limits_{0}^{1}P(x;m,n;\lambda ,p)dx=\sum\limits_{j=0}^{n}\left( 
\begin{array}{c}
n \\ 
j%
\end{array}
\right) ^{p}\lambda ^{j}\frac{\left( 1+j\right) ^{m}-j^{m}}{m+1}.
\label{inP2}
\end{equation}

Combining (\ref{inP-1}) and (\ref{inP2}), we arrive at the following theorem:

\begin{theorem}
\begin{equation}
\sum\limits_{k=0}^{m}\left( 
\begin{array}{c}
m \\ 
k%
\end{array}
\right) \frac{y_{6}(k,n;\lambda ,p)}{m-k+1}=\sum\limits_{j=0}^{n}\left( 
\begin{array}{c}
n \\ 
j%
\end{array}
\right) ^{p}\lambda ^{j}\frac{\left( 1+j\right) ^{m}-j^{m}}{m+1}.
\label{inP8}
\end{equation}
\end{theorem}

Since%
\begin{equation*}
\left( 1+j\right) ^{m}-j^{m}=\sum\limits_{l=0}^{m-1}\left( 
\begin{array}{c}
m \\ 
l%
\end{array}
\right) j^{l},
\end{equation*}%
we modify (\ref{inP8}) as follows:

\begin{corollary}
\begin{equation}
\sum\limits_{k=0}^{m}\left( 
\begin{array}{c}
m \\ 
k%
\end{array}
\right) \frac{m+1}{m-k+1}y_{6}(k,n;\lambda ,p)=\sum\limits_{j=0}^{n}\left( 
\begin{array}{c}
n \\ 
j%
\end{array}
\right) ^{p}\lambda ^{j}\sum\limits_{l=0}^{m-1}\left( 
\begin{array}{c}
m \\ 
l%
\end{array}
\right) j^{l}.  \label{inP8a}
\end{equation}
\end{corollary}

Combining right hand side of equation (\ref{inP8a}) with (\ref{Yp1}) and (%
\ref{y6b}), after some elementary calculations, we get the following
corollary:

\begin{corollary}
\begin{equation*}
P(1;m,n;\lambda ,p)=\frac{m+1}{n!}\sum\limits_{k=0}^{m}\left( 
\begin{array}{c}
m \\ 
k%
\end{array}
\right) \frac{y_{6}(k,n;\lambda ,p)}{m-k+1}+y_{6}(m,n;\lambda ,p).
\end{equation*}
\end{corollary}

\subsubsection{$p$-adic integrals of the polynomials $P(x;m,n;\protect%
\lambda ,p)$}

Here, we give some identities with aid of the $p$-adic integrals of the
polynomials $P(x;m,n;\lambda ,p)$.

In order to give $p$-adic integrals representation of the polynomials $%
P(x;m,n;\lambda ,p)$ on $\mathbb{Z}_{p}$, we need to the following relations
and definitions.

Let $\ \mathbb{K}$ be a field with a complete valuation. Let $f\in C^{1}(%
\mathbb{Z}_{p}\rightarrow \mathbb{\ K)}$, set of continuous derivative
functions. The $p$-adic integral (Volkenborn integral) of $f$ on $\mathbb{Z}%
_{p}$ is given by%
\begin{equation}
\int_{\mathbb{Z}_{p}}f(x)d\mu _{1}(x)=\lim_{N\rightarrow \infty }\frac{1}{
p^{N}}\sum_{x=0}^{p^{N}-1}f(x),  \label{TJM-1}
\end{equation}%
where $\mu _{1}(x)$ denotes the Haar distribution, which is defined by 
\begin{equation*}
\mu _{1}(x)=\mu _{1}(x+p^{N}\mathbb{Z}_{p})=\frac{1}{p^{N}}
\end{equation*}%
(\textit{cf}. \cite{Schikof}, \cite{MSKIMJNT}, \cite{TKIMvolkenborn}).

Kim \cite{Kim2006TMIC} defined the fermionic $p$-adic integral on $\mathbb{Z}%
_{p}$ as follows%
\begin{equation}
\int\limits_{\mathbb{Z}_{p}}f\left( x\right) d\mu _{-1}\left( x\right) = 
\underset{N\rightarrow \infty }{\lim }\sum_{x=0}^{p^{N}-1}\left( -1\right)
^{x}f\left( x\right)  \label{Mmm}
\end{equation}%
where $p\neq 2$ and%
\begin{equation*}
\mu _{-1}(x)=\mu _{-1}\left( x+p^{N}\mathbb{Z}_{p}\right) =\frac{(-1)^{x}}{
p^{N}}
\end{equation*}%
(see also \textit{cf}. \cite{MSKIMJNT}, \cite{TKIMvolkenborn}).

By applying the Volkenborn integral in (\ref{TJM-1}) to both sides of (\ref%
{Yp1}) and (\ref{Yp2}) with respect to $x$ yields, respectively,%
\begin{equation}
\int\limits_{\mathbb{Z}_{p}}P(x;m,n;\lambda ,p)d\mu _{1}\left( x\right)
=\sum\limits_{k=0}^{m}\left( 
\begin{array}{c}
m \\ 
k%
\end{array}
\right) y_{6}(k,n;\lambda ,p)B_{m-k},  \label{inP3}
\end{equation}%
and%
\begin{equation}
\int\limits_{\mathbb{Z}_{p}}P(x;m,n;\lambda ,p)d\mu _{1}\left( x\right)
=\sum\limits_{j=0}^{n}\left( 
\begin{array}{c}
n \\ 
j%
\end{array}
\right) ^{p}\lambda ^{j}B_{m}(j).  \label{inP4}
\end{equation}%
Combining (\ref{inP3}) and (\ref{inP4}), we arrive at the following theorem:

\begin{theorem}
\begin{equation*}
\sum\limits_{k=0}^{m}\left( 
\begin{array}{c}
m \\ 
k%
\end{array}
\right) y_{6}(k,n;\lambda ,p)B_{m-k}=\sum\limits_{j=0}^{n}\left( 
\begin{array}{c}
n \\ 
j%
\end{array}
\right) ^{p}\lambda ^{j}B_{m}(j).
\end{equation*}
\end{theorem}

By applying the fermionic $p$-adic integral in (\ref{Mmm}) to both sides of (%
\ref{Yp1}) and (\ref{Yp2}) with respect to $x$ yields, respectively,%
\begin{equation}
\int\limits_{\mathbb{Z}_{p}}P(x;m,n;\lambda ,p)d\mu _{-1}\left( x\right)
=\sum\limits_{k=0}^{m}\left( 
\begin{array}{c}
m \\ 
k%
\end{array}
\right) y_{6}(k,n;\lambda ,p)E_{m-k},  \label{inP5}
\end{equation}%
and%
\begin{equation}
\int\limits_{\mathbb{Z}_{p}}P(x;m,n;\lambda ,p)d\mu _{-1}\left( x\right)
=\sum\limits_{j=0}^{n}\left( 
\begin{array}{c}
n \\ 
j%
\end{array}
\right) ^{p}\lambda ^{j}E_{m}(j).  \label{inP6}
\end{equation}%
Combining (\ref{inP5}) and (\ref{inP6}), we arrive at the following theorem:

\begin{theorem}
\begin{equation*}
\sum\limits_{k=0}^{m}\left( 
\begin{array}{c}
m \\ 
k%
\end{array}
\right) y_{6}(k,n;\lambda ,p)E_{m-k}=\sum\limits_{j=0}^{n}\left( 
\begin{array}{c}
n \\ 
j%
\end{array}
\right) ^{p}\lambda ^{j}E_{m}(j).
\end{equation*}
\end{theorem}

\subsection{Relations between the polynomials $P(x;m,n;\protect\lambda ,p)$,
Mirimanoff polynomial and Frobenius Euler, Bernoulli polynomials}

Here, we give relations between the polynomials $P(x;m,n;\lambda ,p)$,
Mirimanoff polynomial and Frobenius Euler polynomials. We also give relation
between, the polynomials $P(x;m,n;\lambda ,p)$, sums of powers of
consecutive integers and Bernoulli polynomials and numbers.

In \cite{Vandivier}, Vandiver gave some properties of the Mirimanoff
polynomial, sums of powers of consecutive integers, Bernoulli polynomials
and numbers and also congruence relations for the polynomials.

The Mirimanoff polynomial $f_{m}(x,n)$ is defined by the following forms%
\begin{equation*}
f_{m}(x,n)=\sum\limits_{j=0}^{n-1}j^{m}x^{j}
\end{equation*}
and%
\begin{equation*}
f_{m}(x,n,k)=\sum\limits_{j=0}^{n-1}\left( j+k\right) ^{m}x^{j}
\end{equation*}%
(\textit{cf}. \cite{Vandivier}).

A relation between the Mirimanoff polynomial $f_{n}(x,m)$, the Frobenius
Euler polynomials $H_{m}\left( x;u\right) $ and the polynomials $%
P(x;m,n;\lambda ,p)$ are given by%
\begin{eqnarray*}
f_{n}(x,m) &=&\lambda ^{1-n}P\left( x;m,n;\frac{1}{\lambda },0\right) \\
&=&\sum\limits_{j=0}^{n}\lambda ^{n-j-1}\left( x+j\right) ^{m} \\
&=&\frac{H_{m}\left( x+n;\frac{1}{\lambda }\right) -\frac{1}{\lambda ^{n}}
H_{m}\left( x+n;\frac{1}{\lambda }\right) }{1-\frac{1}{\lambda }}.
\end{eqnarray*}%
The above relation gives us modification of the polynomial $f_{n}(x,m)$ \cite%
[p. 12]{SimsekFPTA}. Carlitz \cite{Carlitz} studied on the Mirimanoff
polynomial $f_{n}(0,m)$ which related to the combinatorial sums, (see also 
\cite{Vandivier}, \cite{Hu}).

Substituting $x=0$ and $p=0$ into (\ref{Yp2}), we have%
\begin{equation*}
P(0;m-1,n-1;\lambda ,0)=\sum\limits_{j=0}^{n-1}\lambda ^{j}j^{m-1}=\frac{
\lambda ^{m}\mathcal{B}_{m}(n;\lambda )-\mathcal{B}_{m}(\lambda )}{m}
\end{equation*}%
(\textit{cf}. \cite{T. Kim}, \cite{Simsekaip}, \cite{SrivastavaChoi2012}).

Substituting $\lambda =1$, $x=0$ and $p=0$ into (\ref{Yp2}), we have%
\begin{equation*}
P(0;m-1,n-1;1,0)=\sum\limits_{j=0}^{n-1}j^{m-1}=\frac{B_{m}(n)-B_{m}}{m}.
\end{equation*}%
About more than 388 years ago, J. Faulhaber discovered the idea of above
formula. Faulhaber gave general formula for the power sum of the first $n$
positive integers (\textit{cf}. \cite{Conway}), see also (\textit{cf}. \cite%
{Comtet}, \cite{Grademir}, \cite{T. Kim}, \cite{KimDkimYsim}, \cite{Moll2012}%
, \cite{Kucukoglu}, \cite{SimsekJDFA}, \cite{Simsekaip}, \cite{srivas18}, 
\cite{SrivastavaChoi2012}, \cite{Temme}).

Substituting $\lambda =-1$, $x=0$ and $p=0$ into (\ref{Yp2}), we have%
\begin{equation*}
P(0;m-1,n-1;-1,0)=\sum\limits_{j=0}^{n-1}(-1)^{j}j^{m-1}=\frac{
(-1)^{n-1}E_{m}(n)-E_{m}}{2}
\end{equation*}%
(\textit{cf}. \cite{Comtet}, \cite{Conway}, \cite{Grademir}, \cite{T. Kim}, 
\cite{KimDkimYsim}, \cite{Kucukoglu}, \cite{Moll2012}, \cite{SimsekJDFA}, 
\cite{Simsekaip}, \cite{srivas18}, \cite{SrivastavaChoi2012}, \cite{Temme}).

\subsection{New family of polynomials including the Michael Vowe polynomial}

In this section we give another new family of polynomials related to the
Michael Vowe polynomial $\mathcal{M}_{n}(x)$ and the Legendre polynomial.

Substituting $\lambda=1$ into (\ref{y6b}), we give the following polynomial $%
R_{n}(x;p)$, which is a polynomial in $x$ of degree $n$: 
\begin{eqnarray*}
R_{n}(x;p) &=&y_{6}(0,n;x,p)=\frac{1}{n!}\sum\limits_{k=0}^{n}\left( 
\begin{array}{c}
n \\ 
k%
\end{array}
\right) ^{p}x^{k} \\
&=&\frac{1}{n!}\left( 
\begin{array}{c}
n \\ 
0%
\end{array}
\right) ^{p}+\frac{1}{n!}\left( 
\begin{array}{c}
n \\ 
1%
\end{array}
\right) ^{p}x+\cdots +\frac{1}{n!}\left( 
\begin{array}{c}
n \\ 
n%
\end{array}
\right) ^{p}x^{n}.
\end{eqnarray*}

\begin{remark}
Substituting $p=2$ into the above equation, we get the Michael Vowe 
polynomial 
\begin{equation*}
\mathcal{M}_{n}(x)=n!R_{n}(x;2)
\end{equation*}
(\textit{cf}. \textup{\cite{Golombek2}}).
\end{remark}

The well-known Euler operator is given by 
\begin{equation*}
\vartheta =x\frac{d}{dx}
\end{equation*}%
(\textit{cf}. \cite[p. 168]{Moll2012}).

By applying the Euler operator to the polynomial $R_{n}(x;p)$, we give the
following sequence $(\alpha _{n}(x;n,p))$:%
\begin{eqnarray*}
\alpha _{1}(x;n,p) &=&x\frac{d}{dx}\left\{ R_{n}(x;p)\right\} , \\
\alpha _{2}(x;n,p) &=&x\frac{d}{dx}\left\{ \alpha _{1}(x;n,p)\right\} , \\
\alpha _{3}(x;n,p) &=&x\frac{d}{dx}\left\{ \alpha _{2}(x;n,p)\right\} ,\ldots
\end{eqnarray*}%
and for $m\geq 2$, we set%
\begin{equation*}
\alpha _{m}(x;n,p)=x\frac{d}{dx}\left\{ \alpha _{m-1}(x;n,p)\right\} .
\end{equation*}

\begin{theorem}
Let 
\begin{equation*}
\alpha _{1}(\lambda ;n,p)=x\frac{d}{dx}\left\{ R_{n}(x;p)\right\} \left\vert
_{x=\lambda }\right. .
\end{equation*}
If $m\geq 2$, then we have 
\begin{equation}
y_{6}(m,n;\lambda ;p)=\alpha _{m}(\lambda ;n,p)  \label{Yp3}
\end{equation}
where 
\begin{equation*}
\alpha _{m}(\lambda ;n,p)=x\frac{d}{dx}\left\{ \alpha _{m-1}(x;n,p)\right\}
\left\vert _{x=\lambda }\right. .
\end{equation*}
\end{theorem}

\begin{proof}
	Proof of this theorem can be easily given by mathematical induction on $m$.
\end{proof}

Substituting $\lambda =1$ into (\ref{Yp3}), we get generalized numbers of
Golombek and Marburg \cite{Golombek2} by the following corollary:

\begin{corollary}
Let 
\begin{equation*}
\alpha _{1}(1;n,p)=x\frac{d}{dx}\left\{ R_{n}(x;p)\right\} \left\vert
_{x=1}\right. .
\end{equation*}
If $m\geq 2$, then we have 
\begin{equation*}
y_{6}(m,n;p)=\alpha _{m}(1;n,p)
\end{equation*}
where 
\begin{equation*}
\alpha _{m}(1;n,p)=x\frac{d}{dx}\left\{ \alpha _{m-1}(x;n,p)\right\}
\left\vert _{x=1}\right. .
\end{equation*}
\end{corollary}

\begin{remark}
Observe that Golombek and Marburg \textup{\cite{Golombek2}} gave the
following  relations: 
\begin{eqnarray*}
Q(n,1) &=&n!\alpha _{1}(1;n,2)=n!x\frac{d}{dx}\left\{ R_{n}(x;2)\right\}
\left\vert _{x=1}\right. , \\
Q(n,2) &=&n!\alpha _{2}(1;n,2)=n!x\frac{d}{dx}\left( x\frac{d}{dx}\left\{
R_{n}(x;2)\right\} \right) \left\vert _{x=1}\right. ,
\end{eqnarray*}
and 
\begin{equation*}
Q(n,3)=n!\alpha _{3}(1;n,2)=n!x\frac{d}{dx}\left( x\frac{d}{dx}\left( x\frac{
d}{dx}\left\{ R_{n}(x;2)\right\} \right) \right) \left\vert _{x=1}\right. .
\end{equation*}
\end{remark}

\begin{remark}
We note that for $\mathcal{M}_{0}(x)=1$ and $\mathcal{M}_{1}(x)=1+x$, Moll  
\textup{\cite{Moll2012}} also gave the following recurrence relations and
identities  for the polynomial $\mathcal{M}_{n}(x)$: 
\begin{equation*}
\mathcal{M}_{n+1}(x)=\frac{2n+1}{n+1}\left( 1+x\right) \mathcal{M}_{n}(x)- 
\frac{n}{n+1}\left( 1-x\right) ^{2}\mathcal{M}_{n-1}(x),
\end{equation*}
(cf. \textup{\cite[p. 169, Theorem 5.3.5.]{Moll2012}}) 
\begin{equation*}
\mathcal{M}_{n}(x)=\left( 1-x\right) ^{n}P_{n}\left( \frac{1+x}{1-x}\right) ,
\end{equation*}
where the Legendre polynomial (cf. \textup{\cite[Eq. (5.3.12)]{Moll2012})}.
Applying the Euler operator to the polynomial $\mathcal{M}_{n}(x)$, Moll 
\textup{\cite[Eq. (5.3.13)]{Moll2012}} gave the following relation:  
\begin{equation*}
M_{j,2}(n)=\vartheta ^{j}\left\{ \mathcal{M}_{n}(x)\right\} ,
\end{equation*}
evaluated at $x=1$.
\end{remark}

\section{Ordinary generating functions for the numbers $y_{6}\left( m,n; 
\protect\lambda ;p\right) $}

In this section, we give some remarks, observations and open questions for
the following function $f\left( t;\lambda ;p,m\right) $:%
\begin{equation*}
\sum_{n=0}^{\infty }y_{6}\left( m,n;\lambda ;p\right) t^{n}=f\left(
t;\lambda ;p,m\right) .
\end{equation*}

We try to compute few special values of the functions $f\left( t;\lambda
;p,m\right) $ as follows.

For $m=0$, and $p=0$, we have%
\begin{equation*}
\sum_{n=0}^{\infty }y_{6}\left( 0,n;\lambda ;0\right) t^{n}=\frac{1}{\lambda
-1}\left( \lambda e^{\lambda t}-e^{t}\right) ,
\end{equation*}%
where%
\begin{equation*}
f\left( t;\lambda ;0,0\right) =\frac{1}{\lambda -1}\left( \lambda e^{\lambda
t}-e^{t}\right) .
\end{equation*}%
For $m=0$, and $p=1$, we have%
\begin{equation*}
\sum_{n=0}^{\infty }y_{6}\left( 0,n;\lambda ;1\right) t^{n}=e^{(\lambda
+1)t},
\end{equation*}%
where%
\begin{equation*}
f\left( t;\lambda ;1,0\right) =e^{(\lambda +1)t}.
\end{equation*}

For $m=0$, $\lambda =1$, and $p=2$, we try to compute the value of the
following series%
\begin{equation*}
\sum_{n=0}^{\infty }y_{6}\left( 0,n;1;2\right) t^{n}=f\left( t;1;2,0\right) .
\end{equation*}%
By using (\ref{y6b}), we have%
\begin{eqnarray*}
\sum_{n=0}^{\infty }y_{6}\left( 0,n;1;2\right) t^{n} &=&\sum_{n=0}^{\infty
}\sum\limits_{k=0}^{n}\left( 
\begin{array}{c}
n \\ 
k%
\end{array}
\right) ^{2}\frac{t^{n}}{n!}=\sum_{n=0}^{\infty }\left( 
\begin{array}{c}
2n \\ 
n%
\end{array}
\right) \frac{t^{n}}{n!} \\
&=&\sum_{n=0}^{\infty }(n+1)C_{n}\frac{t^{n}}{n!}=\sum_{n=0}^{\infty }\frac{
(2n)!}{\left( n!\right) ^{3}}t^{n}
\end{eqnarray*}%
Since%
\begin{equation*}
(2n)!=2^{2n}n!\left( \frac{1}{2}\right) _{n}
\end{equation*}%
(\textit{cf}. \cite{Carlson}, \cite{Seaborn}), we also get%
\begin{equation*}
\sum_{n=0}^{\infty }y_{6}\left( 0,n;1;2\right) t^{n}=\text{ }_{1}F_{1}\left[ 
\begin{array}{c}
\frac{1}{2} \\ 
1%
\end{array}
;4t\right].
\end{equation*}%
Since%
\begin{equation*}
(2n)!=\frac{(-1)^{n}\left( 2\pi \right) ^{2n}}{2\zeta (2n)}B_{2n},
\end{equation*}%
where $\zeta (2n)$ denotes the Riemann zeta function (\textit{cf}. \cite%
{Temme}), we get%
\begin{equation*}
\sum_{n=0}^{\infty }y_{6}\left( 0,n;1;2\right) t^{n}=\sum_{n=0}^{\infty } 
\frac{(-1)^{n}2^{2n-1}\pi ^{2n}B_{2n}}{\left( \Gamma (n+1)\right) ^{3}\zeta
(2n)}t^{n}.
\end{equation*}%
Since%
\begin{equation*}
(2n)!=\frac{4^{n}\Gamma \left( \frac{1+2n}{2}\right) n!}{\sqrt{\pi }},
\end{equation*}%
we also get%
\begin{equation*}
\sum_{n=0}^{\infty }y_{6}\left( 0,n;1;2\right) t^{n}=\frac{1}{\sqrt{\pi }}
\sum_{n=0}^{\infty }\frac{4^{n}\Gamma \left( \frac{1+2n}{2}\right) }{\Gamma
\left( n+1\right) }\frac{t^{n}}{n!}.
\end{equation*}%
Therefore, we obtain the following results for the function $f\left(
t;1;2,0\right) $ as follows:%
\begin{eqnarray*}
f\left( t;1;2,0\right) &=&\text{ }_{1}F_{1}\left[ 
\begin{array}{c}
\frac{1}{2} \\ 
1%
\end{array}
;4t\right] \\
&=&\frac{1}{\sqrt{\pi }}\sum_{n=0}^{\infty }\frac{4^{n}\Gamma \left( \frac{
1+2n}{2}\right) }{\Gamma \left( n+1\right) }\frac{t^{n}}{n!} \\
&=&\sum_{n=0}^{\infty }\left( 
\begin{array}{c}
2n \\ 
n%
\end{array}
\right) \frac{t^{n}}{n!} \\
&=&\sum_{n=0}^{\infty }(n+1)C_{n}\frac{t^{n}}{n!} \\
&=&\sum_{n=0}^{\infty }\frac{(2n)!}{\left( n!\right) ^{2}}\frac{t^{n}}{n!}.
\end{eqnarray*}

\begin{equation*}
\sum_{k=1}^{\infty }y_{6}\left( 0,n;-1;2\right) t^{n}=\sum_{n=0}^{\infty } 
\frac{1}{n!}\sum\limits_{k=0}^{n}(-1)^{k}\left( 
\begin{array}{c}
n \\ 
k%
\end{array}
\right) ^{2}t^{n}.
\end{equation*}%
Combining the above equation with (\ref{AWolf}), we obtain%
\begin{equation*}
\sum_{n=0}^{\infty }y_{6}\left( 0,n;-1;2\right) t^{n}=\sum_{n=0}^{\infty } 
\frac{\sqrt{\pi }2^{n}}{\Gamma \left( \frac{2+n}{2}\right) \Gamma \left( 
\frac{1-n}{2}\right) }\frac{t^{n}}{n!}.
\end{equation*}%
We here do not able to find the explicit formula(s) of the function $%
f(t;\lambda ;p,n)$ in detail, related to all parameters $t$, $\lambda $, $p$%
, and $n$. Therefore, we give the following \textbf{open problems}:

What is the explicit value of the function $f(t;\lambda ;p,n)$?

That is, how can we construct ordinary generating function for the numbers $%
y_{6}\left( m,n;\lambda ;p\right) $?

\section{Identities for the numbers $y_{6}(m,n;\protect\lambda ,p)$}

In this section, we give some functional equations for the generating
functions of the special numbers and polynomials, By using this equations,
we derive some identities and relations including the Bernoulli polynomials
of order $k$, the Euler polynomials of order $k$, the Stirling numbers, and
the numbers $y_{6}(m,n;\lambda ,p)$.

By (\ref{y6a}), we give the following functional equation%
\begin{equation*}
F_{y_{6}}(t,m;\lambda ,p)=\frac{1}{n!}\sum\limits_{k=0}^{n}\left( 
\begin{array}{c}
n \\ 
k%
\end{array}
\right) ^{p}\lambda ^{k}\sum\limits_{l=0}^{k}\left( 
\begin{array}{c}
k \\ 
l%
\end{array}
\right) l!F_{S}(t,l;1).
\end{equation*}%
Combining the above equation with (\ref{SN-1}) and (\ref{y6}), we obtain%
\begin{equation*}
\sum_{m=0}^{\infty }y_{6}(m,n;\lambda ,p)\frac{t^{m}}{m!}=\sum_{m=0}^{\infty
}\frac{1}{n!}\sum\limits_{k=0}^{n}\left( 
\begin{array}{c}
n \\ 
k%
\end{array}
\right) ^{p}\lambda ^{k}\sum\limits_{l=0}^{k}\left( 
\begin{array}{c}
k \\ 
l%
\end{array}
\right) l!S(m,l)\frac{t^{m}}{m!}.
\end{equation*}%
Comparing the coefficients of $\frac{t^{m}}{m!}$ on both sides of the above
equation, after some elementary calculations, we find a relation between the
numbers $y_{6}(m,n;\lambda ,p)$ and the Stirling numbers of the second kind
by the following theorem:

\begin{theorem}
\begin{equation*}
y_{6}(m,n;\lambda ,p)=\sum\limits_{k=0}^{n}\sum\limits_{l=0}^{k}\left( 
\begin{array}{c}
n \\ 
k%
\end{array}
\right) ^{p-1}\frac{S(m,l)}{\left( n-k\right) !\left( k-l\right) !}\lambda
^{k}.
\end{equation*}
\end{theorem}

By (\ref{y6a}), we give the following functional equation%
\begin{equation*}
F_{y_{6}}(t,m;\lambda ,p)=t^{-n}F_{S}(t,n;1)\sum\limits_{k=0}^{n}\left( 
\begin{array}{c}
n \\ 
k%
\end{array}
\right) ^{p}\lambda ^{k}F_{B}(t,j;n).
\end{equation*}

Combining the above equation with (\ref{BerPol.}), (\ref{SN-1}) and (\ref{y6}%
), we get%
\begin{equation*}
\sum_{m=0}^{\infty }y_{6}(m,n;\lambda ,p)\frac{t^{m}}{m!}=\sum_{m=0}^{\infty
}S(m,n)\frac{t^{m}}{m!}\sum_{m=0}^{\infty }\sum\limits_{k=0}^{n}\left( 
\begin{array}{c}
n \\ 
k%
\end{array}
\right) ^{p}\lambda ^{k}B_{m}^{(n)}(j)\frac{t^{m}}{m!}.
\end{equation*}%
Therefore%
\begin{eqnarray*}
&&\sum_{m=0}^{\infty }y_{6}(m,n;\lambda ,p)\frac{t^{m}}{m!} \\
&=&\sum_{m=0}^{\infty }\sum\limits_{k=0}^{n}\sum\limits_{v=0}^{m+n}\frac{
\left( 
\begin{array}{c}
n \\ 
k%
\end{array}
\right) ^{p}\left( 
\begin{array}{c}
m+n \\ 
v%
\end{array}
\right) }{\left( 
\begin{array}{c}
m+n \\ 
n%
\end{array}
\right) n!}\lambda ^{k}S(v,n)B_{m+n-v}^{(n)}(j)\frac{t^{m}}{m!}.
\end{eqnarray*}%
Comparing the coefficients of $\frac{t^{m}}{m!}$ on both sides of the above
equation, after some elementary calculations, we arrive at the following
theorem:

\begin{theorem}
\begin{equation*}
y_{6}(m,n;\lambda ,p)=\sum\limits_{k=0}^{n}\sum\limits_{v=0}^{m+n}\frac{
\left( 
\begin{array}{c}
n \\ 
k%
\end{array}
\right) ^{p}\left( 
\begin{array}{c}
m+n \\ 
v%
\end{array}
\right) }{\left( 
\begin{array}{c}
m+n \\ 
n%
\end{array}
\right) n!}\lambda ^{k}S(v,n)B_{m+n-v}^{(n)}(j).
\end{equation*}
\end{theorem}

By (\ref{y6a}), we give the following functional equation%
\begin{equation*}
F_{y_{6}}(t,m;\lambda ,p)=2^{-n}F_{y_{1}}(t,n;1)\sum\limits_{k=0}^{n}\left( 
\begin{array}{c}
n \\ 
k%
\end{array}
\right) ^{p}\lambda ^{k}F_{E}(t,j;n).
\end{equation*}

Combining the above equation with (\ref{Eul.Pol}), (\ref{ay1}) and (\ref{y6}%
), we get%
\begin{eqnarray*}
&&\sum_{m=0}^{\infty }y_{6}(m,n;\lambda ,p)\frac{t^{m}}{m!} \\
&=&2^{-n}\sum\limits_{k=0}^{n}\left( 
\begin{array}{c}
n \\ 
k%
\end{array}
\right) ^{p}\lambda ^{k}\sum_{m=0}^{\infty }y_{1}(m,n;1)\frac{t^{m}}{m!}
\sum_{m=0}^{\infty }E_{m}^{(n)}(j)\frac{t^{m}}{m!}.
\end{eqnarray*}%
Therefore%
\begin{eqnarray*}
&&\sum_{m=0}^{\infty }y_{6}(m,n;\lambda ,p)\frac{t^{m}}{m!} \\
&=&\sum_{m=0}^{\infty }\sum\limits_{k=0}^{n}\sum\limits_{v=0}^{m}\left( 
\begin{array}{c}
n \\ 
k%
\end{array}
\right) ^{p}\left( 
\begin{array}{c}
m \\ 
v%
\end{array}
\right) \lambda ^{k}y_{1}(v,n;1)E_{m-v}^{(n)}(j)\frac{t^{m}}{m!}.
\end{eqnarray*}%
Substituting (\ref{CC2}) into the above equation, after that comparing the
coefficients of $\frac{t^{m}}{m!}$ on both sides of the above equation,
after some elementary calculations, we arrive at the following theorem:

\begin{theorem}
\begin{equation*}
y_{6}(m,n;\lambda ,p)=\frac{1}{n!2^{n}}\sum\limits_{k=0}^{n}\sum
\limits_{v=0}^{m}\left( 
\begin{array}{c}
n \\ 
k%
\end{array}
\right) ^{p}\left( 
\begin{array}{c}
m \\ 
v%
\end{array}
\right) \lambda ^{k}B(v,n)E_{m-v}^{(n)}(j).
\end{equation*}
\end{theorem}

\subsection*{Acknowledgment}

The present paper was supported by the Scientific Research Project
Administration of Akdeniz University with Project ID:4385.

\section*{References}


\end{document}